\documentclass[11pt, english]{article}
\usepackage[utf8]{inputenc}
\usepackage{amsmath, amssymb}

\usepackage[a4paper, margin=1.2in]{geometry}

\usepackage{babel}
\usepackage[T1]{fontenc}
\usepackage{amsfonts}
\usepackage{amsthm}
\usepackage{enumerate}
\usepackage{graphicx}
\usepackage{pictex}
\usepackage{color}
\usepackage{m_sty}
\usepackage[hidelinks]{hyperref}

\usepackage[dvipsnames]{xcolor}
\usepackage{booktabs}
\allowdisplaybreaks
\usepackage{amsthm}
\usepackage{enumitem} 
\newtheorem{definition}{Definition}
\newtheorem{theorem}{Theorem}
\newtheorem{lemma}{Lemma}
\theoremstyle{remark}

\usepackage[numbers]{natbib}
     \title{Boundary behaviour of potential-type integrals for the multi-term time-fractional diffusion equation}

   \author{
   \vspace{0.3em}
  Karolina Pawlak \\
  \vspace{0.3em} 
  \small Military University of Technology, Faculty of Cybernetics \\
  \vspace{0.3em}
  \small ul. gen. Sylwestra Kaliskiego 2, 00-908 Warsaw 46, Poland \\
  \vspace{0.3em}
  \small E-mail address: karolina.pawlak@wat.edu.pl
}
       
    \date{}

\makeatletter
\renewcommand{\maketitle}{
  \begin{center}
    {\LARGE\bfseries \@title \par}
    \vskip 1em
    {\large
      \lineskip .5em
      \begin{tabular}[t]{c}
        \@author
      \end{tabular}\par}
    \vskip 1em
  \end{center}
}
\makeatother

\begin{document}
 \maketitle
     
     \begin{abstract}
 This paper investigates the boundary behaviour of potential-type integrals for the multi-term time-fractional diffusion equation (MTFDE) across the moving boundary. First, we establish the jump relation for the integral operator associated with the fundamental solution of the inhomogeneous MTFDE. Second, we prove the continuity of the integral operator generated by the kernel corresponding to the homogeneous MTFDE. Krasnoschok obtained similar results for the time-fractional diffusion equation. However, in the multi-term case, the fundamental solution has more complex structure and does not admit standard scaling properties, which requires a different approach. Our results are essential for the analysis of boundary integral equations related to the MTFDE in time-dependent domains. 
     \end{abstract}


     \medskip


\section{Introduction}

In this paper, we study the behaviour of potential-type integrals for the multi-term time-fractional diffusion equation across the moving boundary $x=s(t)$. Layer potential methods play a~fundamental role in the analysis of both classical and fractional differential equations. In \citep[Chapter 1.4 and 5.3]{friedman2}, Friedman discussed the second initial-boundary value problem in domains with Lyapunov boundaries, showing how it can be reduced, with the aid of potentials, to an integral equation. Fabes et al. \cite{fabes} used the method of potentials to construct solutions to Laplace's equation with Dirichlet or Neumann boundary conditions in domains with $C^1$ boundaries. Similar results were obtained for the heat equation in cylindrical domains with $C^1$ boundaries in \cite{fabes2}, and for Lipschitz domains in \cite{brown}.

Kemppainen and Ruotsalainen \cite{kemp7} derived the boundary integral solution of the time-fractional diffusion equation (TFDE) using the single-layer potential. They showed the bijectivity of this operator in a certain range of anisotropic Sobolev spaces. Kemppainen also extended the classical results of Fabes and Brown to TFDE in bounded domains with Lyapunov boundaries \cite{kemp3}, and further to Lipschitz domains \cite{kemp4}. Furthermore, using the boundary integral equation method, he proved the existence and uniqueness of the solution to TFDE on a bounded domain with the Lyapunov boundary and with the Dirichlet or Robin boundary conditions in \citep{kemp5, kemp6}. 

The motivation for writing this paper comes from the work of Avner Friedmann \citep[Chapter 8]{friedman2} and \cite{friedman3}, who proved a~stronger form of a~jump relation across the moving boundary $x=s(t)$ for the potential-type integral corresponding to the heat kernel. The original jump relation was stated earlier by Kolodner \cite{kolodner}. This result was crucial in reducing the classical one dimensional one-phase Stefan problem to an equivalent integral equation for the free boundary. The Stefan problem itself concerns the temperature distribution in a melting solid and the evolution of the moving phase interface $s(t)$. In the one-dimensional one-phase case it can be written as
\begin{equation}\label{stef1}
u_t(x,t) - u_{xx}(x,t) = 0, \qquad 0 < x < s(t), \; 0 < t \leq T,
\end{equation}
\begin{equation}\label{stef2}
u(0,t) = f(t) > 0, \qquad u(s(t),t) = 0, \qquad 0 \leq t \leq T,
\end{equation}
\begin{equation}\label{stef3}
u(x,0) = \varphi(x), \qquad 0 \leq x \leq b, \qquad s(0) = b,
\end{equation}
\begin{equation}\label{stef4}
\dot{s}(t) = -u_x(s(t),t), \qquad 0 < t \leq T,
\end{equation}
where $u$ is the temperature and $s(t)$ describes the motion of the free boundary $x=s(t)$. The reduction of this problem to an integral equation for $\dot{s}(t)$ is based on three steps. 
First, one derives an integral representation  
\[
u(x,t) = \mathcal{U}\!\left(\varphi, f, s, u_x(s(\cdot), \cdot)\right)(x,t)
\]  
in terms of potential-type integrals. This representation naturally decomposes into three parts: the first term reflects the influence of the initial condition, the second term corresponds to the boundary condition at the fixed boundary \(x=0\), and the third term accounts for the Stefan boundary condition \eqref{stef4} at the moving interface \(x=s(t)\). In the next step, one applies the differential operator \(\tfrac{\partial}{\partial x}\) to this representation  to obtain formulas for the spatial derivative \(u_x(x,t)\) expressed in terms of differentiated potential-type integrals. These formulas are not yet equations for \(\dot{s}(t)\): the free boundary velocity emerges only after taking the boundary limit \(x\to s(t)\) and substituting the Stefan condition \eqref{stef4}. Hence the crucial step is the passage to the limit and the identification of the limiting terms, which is exactly where the jump results are used. Finally, we obtain an integral equation for the free boundary
\begin{equation}\label{inteq}
    v(t) =  \mathcal{U}_x\nak{\varphi, f, s, v}(s(t),t), \quad \text{where} \ s(t) = b - \intzt v(s) ds.
\end{equation}
In order to show that the integral equation \eqref{inteq} is equivalent to Stefan problem, one must verify that $s(t)$ given by \eqref{inteq} and $\mathcal{U}\!\left(\varphi, f, s, v\right)(x,t)$ satisfy \eqref{stef1} - \eqref{stef4}.

Krasnoschok generalised this result to the case of the time-fractional diffusion equation. In \cite{MykolaU}  he established the jump relation for the potential-type integral associated with the fundamental solution of inhomogeneous TFDE across the moving boundary $x=s(t)$. Moreover, in \cite{Mykola0} he proved the continuity of the integral operator generated by the kernel related to the initial condition,  across the moving boundary $x=s(t)$. As far as we know, no previous research has investigated the boundary behaviour of layer potentials for the multi-term time-fractional diffusion equation across the moving boundary. We are going to fill this gap.

Multi-term time fractional diffusion equations are used to model complex physical phenomena where a~single-order diffusion equation is insufficient. For instance,  
in the modelling of chloride sub-diffusion in reinforced concrete, Chen et al.~\citep{chen2020multi} demonstrated that a multi-term model provides numerical results that agree better with the real data than other approaches. Moreover, multi-term FDEs have been studied in modelling oxygen delivery through a
capillary to tissues \citep{srivastava}. Finally, for computational modelling, high-order numerical schemes, finite element and finite difference methods are developed for multi-term time-fractional diffusion and diffusion-wave equations, showing their utility in engineering, hydrology, and physics \citep{li,liu}.

Let us now turn to the formulation of the main result of this work. Let $\al_k \in (0,1)$, $\lambda_k > 0$ for $k=1, \ldots, m$ and $\al_1 < \al_2 < \ldots < \al_m$. Kernels $E(x,t)$ and $Z(x,t)$ are fundamental solutions to the following problem
\begin{equation}\label{maine1}
    \sum_{k=1}^{m} \lambda_k D^{\al_k} u(x,t) - u_{xx}(x,t) = f(x,t) , \quad x \in \rr, \ t \in (0,T)
\end{equation}
with the initial condition
\begin{equation}\label{maine2}
    u(x,0) = u_0(x), \quad x \in \rr
\end{equation}
in the sense that
\begin{align*}
    u(x,t) = \intzt \int_{\rr} E(x-y,t-\tau) f(y,\tau) dy d\tau + \int_{\rr} Z(x-y,t) u_0(y) dy .
\end{align*}
The full definition of kernels $E$ and $Z$ and the formulation of the existence theorem will be given in the Preliminaries (see Theorem \ref{fsol}). In this paper, we prove that the potential-type integral corresponding to the kernel $E$ satisfies the following jump relation across the boundary $x=s(t)$.

\begin{theorem}\label{lemma1}
Let $t^{1-\alm}\varphi(t) \in C[0,T]$ and we assume that there exists $\beta>\frac{\alm}{2}$ such that $|s(t)- s(\tau)|\leq |t- \tau|^{\beta}$ for $t, \tau \in [0, T]$. Then for each $t\in (0,T]$ there holds
\begin{align}\label{qwa}
        \lim_{x\rightarrow s(t)^{\mp}} \frac{\partial }{\partial x }  \intzt \varphi(\tau) E(x-s(\tau),t-\tau) \dta =  \pm \jd \varphi(t) + \intzt  \varphi(\tau) E_x(s(t)-s(\tau),t-\tau)\dta .
\end{align}
\end{theorem}
Moreover, we prove the continuity of the integral operator, corresponding to the kernel $Z$, across the boundary $x=s(t)$.
\begin{theorem}\label{lemma2}
    Let $t^{1-\alm}\varphi(t) \in C([0,T])$ and $s \in C([0,T])$. Then for all $t \in (0,T]$ there holds
    \begin{equation}\label{qwa2}
    \begin{split}
        &\lim_{x \to s(t)} \frac{\partial}{\partial x} \intzt \varphi(\tau) Z(x-s(\tau), t-\tau)d\tau = \intzt \varphi(\tau) Z_{x}(s(t)-s(\tau),t-\tau)d\tau.
    \end{split}
    \end{equation}
\end{theorem}
These results are motivated by the study of a generalized fractional Stefan problem, which can be derived by application of the energy conservation law and the nonlocal generalization of Fourier's law. To formulate this problem, we first introduce the distributed-order Caputo derivative
\begin{equation*}
    D^{(\mu)}f(t) = \int_{0}^{1} (D^{\al}f)(t)\mu(\al)d\al, \quad f \in \text{AC}[0,T],
\end{equation*}
where $\mu:[0,1] \to \rr$ is a nonnegative and measurable function and $D^{\al}$ denotes the fractional derivative of Caputo. From \cite{kubica1} we know that if $\mu \in L^1(0,1)$ and $\int_0^1 \mu(\alpha)\,d\alpha > 0$ then then there exists a non-negative $\gm \in L^1_{\text{loc}}[0,\infty)$ such that the operator of fractional integration $I^{(\mu)}$, defined by the formula $I^{(\mu)}f = \gm \ast f$ satisfies
\begin{equation}\label{f7}
    (D^{(\mu)}I^{(\mu)}f)(t) = f(t) \quad \text{for} \quad f \in L^{\infty}(0,T),
\end{equation}
\begin{equation}\label{f8}
    (I^{(\mu)}D^{(\mu)}f)(t) = f(t)-f(0) \quad \text{for} \quad f \in \text{AC}[0,T].
\end{equation}
We focus on the discrete measure $\mu$ defined by  
$\mu(\alpha) = \sum_{k=1}^{m} \lambda_k \delta(\alpha-\alpha_k)$,  
with the same parameters $\lambda_k$ as in \eqref{maine1}.  
In the classical Stefan problem, the function $s(t)$, describing the free boundary, is strictly increasing in time, and we expect the same behaviour in the generalised fractional Stefan problem. It is also worth mentioning what regularity is required for the function $s$ describing the moving boundary to reduce the Stefan problem to an integral equation for the free boundary. For this purpose, we introduce the space
\[
C_{1-\alm}([0,T]) := \left\{ f \in C((0,T]): t^{1-\alm} f(t) \in C([0,T]) \right\}
\]
with the norm
\[
 \|f\|_{C_{1-\alm}([0,T])} = \sup_{t \in [0,T]} |t^{1-\alm} f(t)|,
\]
and
\[
C_{1-\alm}^1([0,T]) := \left\{f \in C([0,T]):  f' \in C_{1-\alm}([0,T]) \right\}
\]
equipped with the norm
\[
\|f\|_{C_{1-\alm}^1([0,T])} = \sup_{t \in [0,T]} |f(t)| + \sup_{t \in [0,T]} |t^{1-\alm} f'(t)|.
\]
We assume that $s \in C^{1}_{1-\alm}([0,T])$. Defining $l(x)$ as the inverse function of the free boundary $s(t)$ on $[s(0),s(T)]$ and setting $l(x)=0$ on $[0,s(0)]$, we introduce the generalized flux
\begin{equation*}
    q^{*}(x,t) := - \frac{d}{dt} I^{(\mu)}_{l(x)} u_{x}(x,t) = - \frac{d}{dt} \int_{l(x)}^{t} g_{\mu}(t-\tau) u_{x}(x,\tau) d\tau.
\end{equation*}
Applying the energy conservation law together with this flux, we obtain the following multi-term time-fractional Stefan problem:
\begin{equation}\label{stefan2}
    \begin{split}
        (i) &\quad   \sum_{k=1}^{m} \lambda_k D^{\al_k}_{l(x)}u(x,t)-u_{xx}(x,t)= 
         - \sum_{k=1}^{m} \lambda_k  \n{t-s^{-1}(x)}^{-\al_k} \frac{1}{\Gamma(1-\al_k)}   \chi_{(s(0),s(t))}(x)  \\
         &\quad \text{in} \ 0 < x < s(t), \ 0 < t \leq T\\
        (ii) &\quad  u_x(0,t)= -h(t) \leq 0, \quad 0 \leq t \leq T\\
        (iii) &\quad  u(s(t),t)=0, \quad 0 \leq t \leq T\\
        (iv) &\quad u(x,0) = u_{0}(x), \quad \text{for} \quad  0 \leq x \leq b, \quad s(0) = b,\\
        (v) &\quad  \dot{s}(t) = - \lim_{x \to s(t)^{-}} \frac{\partial }{\partial t} I^{(\mu)}_{l(x)} u_x(x,t), \quad \text{for} \quad 0 < t \leq T,
    \end{split}
\end{equation} 
where
\begin{align*}
     D^{\al_k}_{l(x)}u(x,t) = \frac{1}{\Gamma(1-\al_k) }\int_{l(x)}^{t} (t-\tau)^{-\al_k} u_t(x,\tau) d\tau \quad k=1,\ldots,m.
\end{align*}
Theorems \ref{lemma1} and \ref{lemma2} are crucial in reducing the multi-term time-fractional Stefan problem to an equivalent integral equation for $s(t)$, following the approach of Friedman outlined above.

In our analysis, we will follow the technique used by Friedmann in \cite{friedman2} and Krasnoschok in \citep{Mykola0, MykolaU}. However, in multi-term case the fundamental solution has a more complex structure and does not admit standard scaling properties, which generate additional difficulties to overcome.

The paper is organised as follows. In Chapter 2, we will briefly recall the results from \cite{pskhu3} that play a significant role in the following part of the paper. Chapter 3 is devoted to the estimates for the $E$ and $Z$ kernels. In Chapter 4, we
prove the jump relation for the potential-type integral corresponding to the kernel $E$ across the boundary $x=s(t)$. Finally, in Chapter 5, we examine the limiting behaviour of the integral operator corresponding to the kernel $Z$.

\section{Preliminaries}

In this section we will present important results concerning the fundamental solutions to the fractional diffusion equation with discretely distributed
differentiation operator. These results mainly comes from \cite{pskhu3}.

\begin{definition}
Let $T > 0$ and $\al \in (0,1)$.
For every $f \in L^1(0,T)$ the Riemann-Liouville fractional integral of order $\al$ is defined by
    \begin{equation}\label{f1}
    I^{\al}f(t) = \jgja \intzt (t-\tau)^{\al-1} f(\tau) d\tau.
\end{equation}
Moreover, for $f \in L^1(0,T)$ we define
\begin{align}\label{f22}
    I^{0}f(t) = f(t).
\end{align}
For $f$ regular enough, we define the Riemann-Liouville fractional derivative as
\begin{equation*}
    \partial^{\al}f(t) = \frac{d}{dt} I^{1-\al}f(t) = \frac{1}{\Gamma(1-\al)} \frac{d}{dt} \intzt (t-\tau)^{-\al} f(\tau) d\tau
\end{equation*}
and the fractional derivative of Caputo as
\begin{equation*}
     D^{\al}f(t) = \frac{d}{d t} \n{I^{1-\al}[f(t)-f(0)]} = \frac{1}{\Gamma(1-\al)} \frac{d}{dt} \intzt (t-\tau)^{-\al} [f(\tau) - f(0)] d\tau.
\end{equation*}
We note that if $f \in \text{AC}[0,T]$, then $D^{\al}f$ may be equivalently written in the form
\begin{equation*}
   D^{\al}f(t) = I^{1-\al}f'(t) = \frac{1}{\Gamma(1-\al)} \intzt (t-\tau)^{-\al} f'(\tau) d\tau.
\end{equation*}
\end{definition}

\begin{definition}[\cite{luchko}, p. 3]\label{defW}
The Wright function is defined by
\begin{equation}\label{wf}
         W(z; -\beta, \delta) = \sum_{k=0}^{\infty} \frac{z^k}{k! \Gamma(\delta -\beta k)}, \quad z \in \C, \ \beta < 1, \ \delta \in \C.
\end{equation}

For $\beta < 1$, the series on the right hand side of the formula \eqref{wf} is convergent for all $z \in \C$. It is also convergent for $\beta = 1$ and $|z| < 1$ and for $\beta = 1$ and $|z| = 1$ under the condition $\re \delta > 1$. However, the Wright function is an entire function only in the case $\beta < 1$ and thus this condition is usually included in its definition. 

\end{definition}

\begin{lemma}[\cite{pskhu2}, \S 2, \S 3, \cite{MFbook}, equality (7.2.3)]  
Let $\beta \in (0,1)$.  Then the next assertions follow:
\medskip
\begin{enumerate}[label=(\alph*)] 
    \item For all $\nu, c, t > 0$ and $\delta \in \rr$ we have
    \medskip
      \begin{equation}\label{der1}
         I^{\nu} t^{\delta-1} W(-ct^{-\beta};-\beta,\delta) = t^{\delta+\nu-1} W(-ct^{-\beta};-\beta,\delta+\nu),
    \end{equation}  
   \begin{equation*}
       \partial^{\nu} t^{\delta-1} W(-ct^{-\beta};-\beta,\delta) = t^{\delta-\nu-1} W(-ct^{-\beta};-\beta,\delta-\nu),
    \end{equation*}
    \begin{align}\label{qwl}
    \intzn z^{\nu-1} W(-z;-\beta,\delta) dz = \frac{\Gamma(\nu)}{\Gamma(\beta \nu +\delta)}.
        \end{align}
    \item For all $z > 0$  
 \begin{align}\label{qwl1}
    W(-z;-\beta, \delta) > 0, \quad \quad \delta \geq 0,
\end{align}
\begin{equation}\label{qwj}
    W(-z; -\beta, 0)  = z\beta W(-z; -\beta, 1-\beta). 
\end{equation}
\end{enumerate}
\end{lemma}

Throughout this paper we will often use the following inequality
\begin{equation}\label{expe}
    |y|^{\gamma}\exp\n{-\kappa|y|^{\frac{2}{2-\al}}} \leq C(\gamma,\kappa, \al), \quad y>0, \ \gamma \geq 0, \ \kappa > 0.
\end{equation}

We define (see \cite{pskhu3}, equality (6) i (7))
\begin{multline}\label{fmt}
    \Gamma^{\mu}_1(x,t; \al_1, \ldots, \al_m; \lambda_1, \ldots, \lambda_m) \\
    =  (4\pi)^{-\jd} \intzn p^{-\jd} e^{-\frac{|x|^2}{4p} } S^{\mu}_m(t;-\lambda_1 p, \ldots, -\lambda_m p;-\al_1,\ldots, -\al_m)   dp,
\end{multline}
where for any positive parameters $\lf_1 , \ldots, \lf_m$ there holds
\begin{align}\label{smim}
 S^{\mu}_m(t;-\lf_1 , \ldots, -\lf_m ;-\al_1,\ldots, -\al_m)  = (h_1 \ast h_2 \ast \ldots \ast h_m)(t),
\end{align}
and
\begin{align}\label{smim1}
    h_j(t) = t^{\mu_j-1} W(-\lf_j t^{-\al_j};-\al_j,\mu_j) \quad \text{for} \ j=1,\ldots, m,
\end{align}
where $\mu_j \in \rr$ such that $\mu = \sum_{j=1}^{m} \mu_j$. To simplify the notation, we will write
\begin{align*}
     \Gamma^{\mu}_1(x,t) = \Gamma^{\mu}_1(x,t; \al_1, \ldots, \al_m; \lambda_1, \ldots, \lambda_m),
\end{align*}
\begin{align}\label{wmpt}
     w_{\mu}(p,t) = S^{\mu}_m(t;-\lambda_1 p , \ldots, -\lambda_m p;-\al_1,\ldots, -\al_m) .
\end{align}

\begin{remark}[\cite{pskhu4}, Remark 1]\label{insmm}
    The function $ S^{\mu}_m(t;-\lambda_1 , \ldots, -\lambda_m ;-\al_1,\ldots, -\al_m)$ is independent of the distribution of values of $\mu_i \in \rr$, but depends only on their sum $\mu$.
\end{remark}

\begin{remark}
    We observe that for any positive parameters $\lf_1, \ldots, \lf_m$ and $t>0$ we have
\begin{align}\label{ej01}
    S^{0}_m(t;-\lf_1 ,\ldots, -\lf_m ;-\al_1,\ldots, -\al_m) > 0 
\end{align}
because using property \eqref{qwj} and \eqref{qwl1} we have $h_j(t) > 0$ for $t > 0$ and $j = 1,\ldots,m$.
\end{remark}

\begin{remark}
    From Lemma 1 in \cite{pskhu3} we know that for all $\mu \in \rr$, $\kappa < (1-\alm)\n{\alm^{\alm} \lamm}^{\frac{1}{1-\alm}}$, $p > 0$ and $t > 0$ we have the estimate
        \begin{align}\label{wmpt4}
            |w_{\mu}(p,t)| \leq C t^{\mu-1} \n{pt^{-\alm}}^{-\theta} \exp\n{-\kappa \n{\frac{p}{t^{\alm}}   }^{\frac{1}{1-\alm}}}, 
        \end{align}
        where 
\begin{align*}
     \theta \geq \left\{ \begin{array}{ll}
0, & \ (-\mu) \notin \N_{0},\\
-1, & (-\mu) \in \N_{0},
\end{array} \right.
\end{align*}        
   and  $C = C(\mu, \lamm, \alm, \kappa, \theta)$. This estimate will be very useful for future analysis.
\end{remark}

We define the kernels 
\begin{align}\label{zmt}
    Z(x,t) = \sum_{k = 1}^{m} \lambda_k \Gamma^{1-\al_k}_{1}(x,t),
\end{align}
and
\begin{align}\label{emt}
    E(x,t) = \Gamma_1^{0}(x,t),
\end{align}
which are the fundamental solutions to the problem \eqref{maine1}-\eqref{maine2}. More precisely:
\begin{itemize}
    \item The kernel $Z(x,t)$ satisfies, in the sense of distributions, the following homogeneous equation
    \begin{equation*}
        \sum_{k=1}^{m} \lambda_k D^{\alpha_k} Z(x,t) -  Z_{xx}(x,t) = 0, 
        \quad t>0, \ x \in \mathbb{R},
    \end{equation*}
    with the initial condition $Z(x,0)=\delta(x)$.
    \item The kernel $E(x,t)$ satisfies, in the sense of distributions, the following  inhomogeneous equation 
    \begin{equation*}
        \sum_{k=1}^{m} \lambda_k D^{\alpha_k} E(x,t) -  E_{xx}(x,t) 
        = \delta(t)\delta(x), 
        \quad t>0, \ x \in \mathbb{R},
    \end{equation*}
    with $E(x,0)=0$.
\end{itemize}

\begin{definition}[\cite{pskhu3}, \S 2]
    Let us denote
\[
D = \mathbb{R} \times (0, T), \quad D_0 = \mathbb{R} \times [0, T).
\]

A \emph{regular solution} of equation \eqref{maine1} in the domain \(D\) is a function \(u(x,t)\) such that:

\begin{enumerate}
    \item \(u(x,t)\) is twice continuously differentiable with respect to the variable \(x\) in \(D\);
    \item \(t^{1-\gamma} u(x,t) \in C(D_0)\) for some \(\gamma > 0\);
    \item The function
    \[
    u(x,t) \in C(D_0),
    \]
 is absolutely continuous with respect to \(t\) on the half-interval \([0, T)\) for each fixed \(x \in \mathbb{R}\).
\end{enumerate}
\end{definition}

This allows us to establish the subsequent theorem.

\begin{theorem}[see \cite{pskhu3}, Theorem 1]\label{fsol}
Let $u_{0}(x)$ be continuous in $\rr$, $t^{1-\gamma}f(x,t) \in C(\rr \times [0,T))$   for some $\gamma > 1-\alm$ and there exists exponent $q>0$ and $C>0$ such that
for all $t \in (0,T)$ we have
\begin{align*}
    |f(x_1,t)-f(x_2,t)| \leq C |x_1-x_2|^q \quad \text{for} \quad x_1, x_2 > 0.
\end{align*}
Moreover, we assume that
\begin{align*}
    \lim_{|x| \to \infty} u_{0}(x) \exp\n{-\kappa |x|^{\frac{2}{2-\alm}}} = 0,
\end{align*}
\begin{align*}
   \lim_{|x| \to \infty} \sup_{t \in (0,T)} t^{1-\gamma} f(x,t) \exp\n{-\kappa |x|^{\frac{2}{2-\alm}}} = 0,
\end{align*}
where $\kappa < (2-\alm)\n{\frac{\alm}{T}}^{\frac{\alm}{2-\alm}} \n{\frac{\lamm}{4}}^{\frac{1}{2-\alm}}$. Then, function $u$ defined as
\begin{align*}
    u(x,t) = \intzt \int_{\rr} E(x-y,t-\tau) f(y,\tau) dy d\tau + \int_{\rr} Z(x-y,t) u_0(y) dy 
\end{align*}
is the regular solution to problem
\begin{equation*}
\begin{split}
    \sum_{k=1}^{m} \lambda_k D^{\al_k} u(x,t) - u_{xx}(x,t) &= f(x,t) , \quad x \in \rr, \ t \in (0,T)\\
   u(x,0) &= u_0(x), \quad x \in \rr.
\end{split}
\end{equation*}
\end{theorem}

\begin{lemma}[see \cite{pskhu3}, Lemma 4]\label{lemest}
For $\mu \in \rr$, $\kappa < \n{2-\alm}\n{\frac{\alm^{\alm}\lamm}{4}}^{\frac{1}{2-\alm}}$, $|x|>0$ and $t > 0$ there holds
    \begin{equation}\label{gam1}
        \dm{\Gamma^{\mu}_{1}(x,t)} \leq C t^{\mu+\aldm-1}  \exp\n{-\kappa z^{\frac{1}{2-\alm}} },
    \end{equation}
    \begin{equation}\label{gam2}
        \dm{\frac{\partial}{\partial x} \Gamma^{\mu}_{1}(x,t)} \leq C |x| t^{\mu-\aldm  -1} \eta(z,p+2) \exp\n{-\kappa z^{\frac{1}{2-\alm}} },
    \end{equation}
    \begin{equation}\label{gam3}
         \dm{\frac{\partial^2}{\partial x^2} \Gamma^{\mu}_{1}(x,t)}\leq C t^{\mu-\aldm-1} \eta(z,p+2) \exp\n{-\kappa z^{\frac{1}{2-\alm}} },
    \end{equation}
where $z = |x|^2t^{-\alm}$, $C = C(\mu,\lamm, \alm, \kappa)$,
\begin{align*}
    p = \left\{ \begin{array}{ll}
1 & \textrm{for \quad $(-\mu) \notin \N_{0},$}\\
-1 & \textrm{for \quad $(-\mu) \in \N_{0},$}
\end{array} \right. \quad \quad
\eta(z,n) = \left\{ \begin{array}{ll}
1 & \textrm{for \quad $n<2,$}\\
1+|\ln z| & \textrm{for \quad $n=2$,}\\
z^{1-\frac{n}{2}} & \textrm{for \quad $n>2$.}
\end{array} \right.
\end{align*}
Moreover, for $t \in (0,T]$ there holds
\begin{equation}\label{gam4}
         \dm{\Delta_x \Gamma^{\mu}_{1}(x,t)}\leq C t^{\mu-\aldm-1}  \exp\n{-\kappa z^{\frac{1}{2-\alm}} },
    \end{equation}
    where  $C = C(T, \mu, \lamm, \al_1, \ldots, \alm, \kappa)$.
\end{lemma}

\begin{remark}\label{corest}
We can notice that there is a slight difference between the formulation of inequality \eqref{gam4} and its counterpart in Professor Pskhu's work, namely inequality $(57)$ in Lemma 4 in \cite{pskhu3}. In fact, the equality \eqref{gam4} holds for $t \in (0,T]$ and the constant $C$ depends also on $T$ and $\al_1, \ldots, \al_{m-1}$. This represents the corrected form of the inequality $(57)$ of Lemma 4 in \cite{pskhu3}. Let us justify it. From the following equalities (see $(44)$ and $(47)$ in \cite{pskhu3}), which, in case $n=1$, hold for $x \in \rr$ and $t > 0$,
\begin{align*}
    &\partial^{\alpha_k} \Gamma^{\mu}_{1}(x,t) = \Gamma^{\mu-\alpha_k}_1(x,t), \quad k=1,\ldots,m,\\
    &\left( \sum_{k=1}^{m} \lambda_k \partial^{\alpha_k} - \Delta_x \right) \Gamma^{\mu}_{1}(x,t) = 0,
\end{align*}
it follows that
\begin{align}\label{ap0}
    \Delta_x \Gamma^{\mu}_{1}(x,t) = \sum_{k=1}^{m} \lambda_k \Gamma^{\mu-\alpha_k}_1(x,t), 
    \quad x \in \rr, \ t > 0.
\end{align}
Moreover, using \eqref{gam1} we derive for all k = 1,\ldots, m
\begin{align}\label{ap}
    |\Gamma^{\mu-\al_k}_1(x,t)| \leq C_k t^{\mu-\al_k+\aldm -1} \exp\n{-\kappa z^{\frac{1}{2-\alm}}} \quad |x| > 0, \ t > 0,
\end{align}
where $C_k = C_k(\mu, \al_k, \lamm, \alm, n, \kappa)$.
Hence, using \eqref{ap0} and \eqref{ap} we get
\begin{align}\label{ap1}
     |\Delta_x \Gamma^{\mu}_{1}(x,t)| \leq  \sum_{k=1}^{m} C_k \lambda_k t^{\mu-\al_k+\aldm-1} \exp\n{-\kappa z^{\frac{1}{2-\alm}}}, \quad |x| > 0, \ t > 0       .
\end{align}
For $t \in (0,T]$ and for all $k=1,\ldots,m$ we have
\begin{align*}
    t^{-\al_k} = T^{-\al_k} \n{\frac{t}{T}}^{-\al_k} \leq T^{-\al_k} \n{\frac{t}{T}}^{-\alm} = T^{\alm-\al_k} t^{-\alm}.
\end{align*}
Using this observation in \eqref{ap1} we get correct version of inequality $(57)$ from \cite{pskhu3}
\begin{align*}
     |\Delta_x \Gamma^{\mu}_{1}(x,t)| &\leq C(T,\mu, \lamm, \al_1, \ldots, \alm,  \kappa) t^{\mu-\aldm-1} \exp\n{-\kappa z^{\frac{1}{2-\alm}}}, \quad \quad |x| > 0, \ t \in (0,T].
\end{align*}
\end{remark}

In Chapter 4. we will need the following version of Leibniz integral rule.

\begin{theorem}[see \cite{folland}, Theorem 2.27]\label{lir}
   Let $X$ be an open subset of $\rr$ and $\Omega$ be a measure space. Suppose $f: X \times \Omega \to \rr$ satisfies the following conditions:
   \begin{enumerate}
      \item $f(x,\omega)$ is Lebesgue-integrable function of $\omega$ for each $x \in X$.
      \item For almost all $\omega \in \Omega$, the partial derivative $f_x$ exists for all $x \in X$.
        \item There is an integrable function $\theta: \Omega \to \rr$ such that $|f_x(x,\omega)| \leq \theta(\omega)$ for all $x \in X$ and almost every $\omega \in \Omega$.
  \end{enumerate}
   Then, for all $x \in X$,
   \begin{align*}
       \frac{d}{dx} \int_{\Omega} f(x,\omega) d\omega = \int_{\Omega} f_x(x,\omega) d\omega.
   \end{align*}
\end{theorem}

\section{Estimates for kernels E and Z}
Estimating the E and Z kernels is crucial for continuing research. The required estimates will be derived from Lemma \ref{lemest}. 

\begin{lemma}
Let $\kappa < (2-\alm)\n{\frac{\alm^{\alm} \lamm}{4}}^{\frac{1}{2-\alm}}$, $|x|>0$ and $t > 0$. Then, we have
    \begin{equation}\label{ls1}
        \dm{E(x,t)} \leq C t^{\aldm-1} \exp\n{-\kappa \n{\frac{|x|}{t^{\aldm}}   }^{\frac{2}{2-\alm}} },
    \end{equation}
    \begin{equation}\label{ls2}
        \dm{E_x(x,t)} \leq C |x| t^{-\aldm -1} \exp\n{-\kappa \n{\frac{|x|}{t^{\aldm}}   }^{\frac{2}{2-\alm}} },
    \end{equation}
    \begin{equation}\label{ls3}
        \dm{E_{xx}(x,t)} \leq C t^{-\aldm-1}\exp\n{-\kappa \n{\frac{|x|}{t^{\aldm}}   }^{\frac{2}{2-\alm}} },
    \end{equation}
where $C = C(\lamm,\alm,\kappa)$.
\end{lemma}

\begin{proof}
It is a straightforward consequence of Lemma \ref{lemest} if we take into account representation \eqref{emt} of kernel $E(x,t)$.
\end{proof}

\begin{lemma}
Let $\kappa < (2-\alm)\n{\frac{\alm^{\alm} \lamm}{4}}^{\frac{1}{2-\alm}}$, $|x|>0$ and $t \in (0,T]$. Then, we have
\begin{equation}\label{ls4}
        \dm{Z(x,t)} \leq C  t^{-\aldm} \exp\n{-\kappa \n{\frac{|x|}{t^{\aldm}}   }^{\frac{2}{2-\alm}} },
    \end{equation}
    \begin{equation}\label{ls5}
        \dm{Z_x(x,t)} \leq C   t^{-\alm} \exp\n{-\kappa \n{\frac{|x|}{t^{\aldm}}   }^{\frac{2}{2-\alm}} },
    \end{equation}
    \begin{equation}\label{ls6}
        \dm{Z_{xx}(x,t)} \leq C   t^{-\trd \alm} \exp\n{-\kappa \n{\frac{|x|}{t^{\aldm}}   }^{\frac{2}{2-\alm}} },
    \end{equation}
    where $C = C(T, \al_1, \ldots, \al_m, \lamm, \kappa)$.
\end{lemma}

\begin{proof}
From \eqref{zmt} and \eqref{gam1} we obtain
\begin{align}\label{num1}
    |Z(x,t)| &\leq \sum_{k=1}^{m} \lambda_k \dm{\Gamma_{1}^{1-\al_k}(x,t)} \leq \sum_{k=1}^{m} \lambda_k C_k t^{-\al_k+\aldm}   \exp\n{-\kappa \n{\frac{|x|}{t^{\aldm}}   }^{\frac{2}{2-\alm}} },
\end{align}
where $C_k = C_k(\al_k,\lamm, \alm, \kappa)$ 
for each $k=1, \ldots, m$. Since $t \in (0,T]$ then 
\begin{align}\label{num2}
    t^{-\al_k} = \n{\frac{t}{T}}^{-\al_k} T^{-\al_k} \leq \n{\frac{t}{T}}^{-\alm} T^{-\al_k} = T^{\alm-\al_k} t^{-\alm}.
\end{align}
Hence, using the estimate \eqref{num2} in \eqref{num1} we get
\begin{align*}
    |Z(x,t)| &\leq \sum_{k=1}^{m} \lambda_k C_k T^{\alm-\al_k} t^{-\aldm}   \exp\n{-\kappa \n{\frac{|x|}{t^{\aldm}}   }^{\frac{2}{2-\alm}} }\\
    &= C t^{-\aldm}   \exp\n{-\kappa \n{\frac{|x|}{t^{\aldm}}   }^{\frac{2}{2-\alm}} },
\end{align*}
where $C = C(T, \al_1, \ldots, \al_m, \lamm, \kappa)$. Thus, we have \eqref{ls4}. Furthermore, from \eqref{zmt} and \eqref{gam2} we have
\begin{align*}
    |Z_x(x,t)| &\leq \sum_{k=1}^{m} \lambda_k \dm{\frac{\partial}{\partial x}\Gamma_{1}^{1-\al_k}(x,t)} \\
    &\leq \sum_{k=1}^{m} \lambda_k C_k |x| t^{-\al_k-\aldm} \n{\frac{|x|^2}{t^{\alm}}}^{-\jd}  \exp\n{-\kappa \n{\frac{|x|}{t^{\aldm}}   }^{\frac{2}{2-\alm}} }\\
    &= \sum_{k=1}^{m} \lambda_k C_k t^{-\al_k}   \exp\n{-\kappa \n{\frac{|x|}{t^{\aldm}}   }^{\frac{2}{2-\alm}} },
\end{align*}
where $C_k = C_k(\al_k,\lamm, \alm, \kappa)$ for each $k=1,\ldots,m$.
Since $t \in (0,T]$ then using \eqref{num2} we get
\begin{align*}
    |Z_x(x,t)| &\leq \sum_{k=1}^{m} \lambda_k C_k T^{\alm-\al_k} t^{-\alm}  \exp\n{-\kappa \n{\frac{|x|}{t^{\aldm}}   }^{\frac{2}{2-\alm}} }\\
    &=C t^{-\alm}  \exp\n{-\kappa \n{\frac{|x|}{t^{\aldm}}   }^{\frac{2}{2-\alm}} },
\end{align*}
where $C = C(T, \al_1, \ldots, \al_m, \lamm, \kappa)$. Thus, we obtain \eqref{ls5}. Finally, from \eqref{zmt} and \eqref{gam4} we have
\begin{align*}
    |Z_{xx}(x,t)| &\leq \sum_{k=1}^{m} \lambda_k \dm{\frac{\partial^2}{\partial x^2}\Gamma_{1}^{1-\al_k}(x,t)} \\
    &\leq \sum_{k=1}^{m} \lambda_k C_k  t^{-\al_k-\aldm}  \exp\n{-\kappa \n{\frac{|x|}{t^{\aldm}}   }^{\frac{2}{2-\alm}} },
\end{align*}
where $C_k = C_k(T,\al_1, \ldots, \al_m,\lamm, \kappa)$ for each $k=1,\ldots,m$. Again using \eqref{num2} we get
\begin{align*}
    |Z_{xx}(x,t)| &\leq C t^{-\trd \alm}  \exp\n{-\kappa \n{\frac{|x|}{t^{\aldm}}   }^{\frac{2}{2-\alm}} },
\end{align*}
where $C = C(T, \al_1, \ldots, \al_m, \lamm, \kappa)$. This gives us \eqref{ls6}.
\end{proof}

\section{Jump relation for integral operator corresponding to kernel E}

In this chapter, we are going to prove Theorem \ref{lemma1}. First, we show an auxiliary lemma that allows us to differentiate under the integral sign in \eqref{qwa}.

\begin{lemma}\label{duis}
    Let $t^{1-\alm}\varphi(t)\in C[0,T]$. Assume that there exists $\beta>\frac{\alm}{2}$ such that $|s(t)- s(\tau)|\leq |t- \tau|^{\beta}$ for $t, \tau \in [0, T]$. Then, for all $t\in (0,T]$ and $x \neq \st$ there holds
\eqq{ \frac{\partial }{\partial x }  \intzt \varphi(\tau) E(x-s(\tau),t-\tau) \dta =   \intzt  \varphi(\tau) E_x(x-s(\tau),t-\tau)\dta .}{lqwa}
\end{lemma}

\begin{proof}
We use Theorem \ref{lir} in order to prove \eqref{lqwa}. First, we show that $\varphi(\tau) E(x-s(\tau),t-\tau) $ is a~Lebesgue integrable function of $\tau$ for each $t \in (0,T]$ and $x \in \rr$ such that $x \neq \st$. In fact, using \eqref{ls1} and the assumption on $\varphi$, we get
\begin{align*}
    \intzt \dm{\varphi(\tau) E(x-s(\tau),t-\tau)} \dta &\leq C \sup_{t \in (0,T)}|t^{1-\alm}\varphi(t)| \intzt \tau^{\alm-1} \n{t-\tau}^{\aldm-1}  \dta \\
    &= C \sup_{t \in (0,T)}|t^{1-\alm}\varphi(t)| \cdot t^{\trd \alm -1} B\n{\alm, \aldm} < +\infty.
\end{align*}
According to point 3. of Theorem \ref{lir} we must find an integrable function $g: (0,T) \to \rr$ such that 
\begin{align*}
    |\varphi(\tau) E_x(x-s(\tau),t-\tau)| \leq g(\tau) \quad \text{for almost every} \ \tau \in (0,T).
\end{align*}
Since $x \neq \st$, we can set $\omega:= |\st - x| > 0$. From the continuity of function $s$ we know that for all $\ep > 0$ there exists $\delta > 0$ such that for every $\tau \in (t-\delta, t+\delta)$ we have $|\sta - \st| < \ep$.
Hence, 
\begin{itemize}
    \item for $x < \st$ we have $\omega = \st-x$ and $\sta - x > \st - x - \ep = \omega - \ep$,
    \item for $x > \st$ we have $\omega = x-\st$ and $x-\sta > x-\st - \ep = \omega - \ep$.
\end{itemize}
Thus, taking $\ep \in (0,\omega)$, for example $\ep = \frac{\omega}{2}$, we see that there exists $\delta > 0$ such that for all $\tau \in (t-\delta, t+\delta)$ there is $|x-\sta|> \frac{\omega}{2}$. Hence, we can split the following integral into two ones
\begin{align*}
    \frac{\partial }{\partial x }  \intzt  \varphi(\tau) E(x-s(\tau),t-\tau) \dta &= \frac{\partial }{\partial x }  \int_{0}^{t-\delta} \varphi(\tau) E(x-s(\tau),t-\tau) \dta\\
    &\;\;\;+ \frac{\partial }{\partial x }  \int_{t-\delta}^{t} \varphi(\tau) E(x-s(\tau),t-\tau) \dta.
\end{align*}
We construct the majorant separately on both intervals.
\begin{enumerate}
    \item Majorant on $(0,t-\delta)$:
Using \eqref{ls2}, the assumption on $\varphi$ and \eqref{expe} we obtain
\begin{align*}
   &| \varphi(\tau) E_{x}(x-s(\tau) ,t-\tau)| \\
   &\leq  C \sup_{t \in (0,T)}|t^{1-\alm}\varphi(t)| \tau^{\alm-1} \frac{|x-s(\tau) |}{\n{t-\tau}^{\aldm+1}} \exp\n{-\kappa \n{\frac{|x-s(\tau) |}{\n{t-\tau}^{\aldm}}}^{\frac{2}{2-\alm}}} \\
   &\leq C \sup_{t \in (0,T)}|t^{1-\alm}\varphi(t)| \frac{\tau^{\alm-1}}{t-\tau} .
\end{align*}
Hence, let $g_1(\tau) := C \sup_{t \in (0,T)}|t^{1-\alm}\varphi(t)| \frac{\tau^{\alm-1}}{t-\tau}$ be the majorant on the interval $(0,t-\delta)$. Function $g_1$ is integrable on $(0,t-\delta)$, because
\begin{align*}
    \int_{0}^{t-\delta} |g_1(\tau)| \dta &=  C \sup_{t \in (0,T)}|t^{1-\alm}\varphi(t)|  \int_{0}^{t-\delta} \frac{\tau^{\alm-1}}{t-\tau} \dta \\
    &\leq C \sup_{t \in (0,T)}|t^{1-\alm}\varphi(t)|  \int_{0}^{t-\delta} \frac{\tau^{\alm-1}}{\delta} \dta \\
    &=  C \sup_{t \in (0,T)}|t^{1-\alm}\varphi(t)|  \frac{(t-\delta)^{\alm}}{\alm \delta}  \\
    &\leq C \sup_{t \in (0,T)}|t^{1-\alm}\varphi(t)|  \frac{t^{\alm}}{\alm \delta}.
\end{align*}
\item Majorant on $(t-\delta,t)$. Again using \eqref{ls2}, the assumption on $\varphi$ and \eqref{expe} we get
    \begin{align*}
   &| \varphi(\tau) E_{x}(x-s(\tau),t-\tau)| \\
   &\leq  C \sup_{t \in (0,T)}|t^{1-\alm}\varphi(t)| \tau^{\alm-1} \frac{|x-s(\tau) |}{\n{t-\tau}^{\aldm+1}} \exp\n{-\kappa \n{\frac{|x-s(\tau)|}{\n{t-\tau}^{\aldm}}}^{\frac{2}{2-\alm}}} \\
   &=  C \sup_{t \in (0,T)}|t^{1-\alm}\varphi(t)| \tau^{\alm-1} |x-\sta|^{-\frac{1}{\aldm} }  \\
   &\;\;\; \times \n{\frac{|x-s(\tau)|}{\n{t-\tau}^{\aldm}}}^{\frac{\aldm+1}{\aldm}} \exp\n{-\kappa \n{\frac{|x-s(\tau)|}{\n{t-\tau}^{\aldm}}}^{\frac{2}{2-\alm}}} \\
   &\leq C \sup_{t \in (0,T)}|t^{1-\alm}\varphi(t)| \tau^{\alm-1} \frac{1}{|x-\sta|^{\frac{2}{\alm}}     }  \\
   &\leq C \sup_{t \in (0,T)}|t^{1-\alm}\varphi(t)| \tau^{\alm-1} \frac{1}{\n{\frac{\omega}{2}}^{\frac{2}{\alm}}     } .
\end{align*}
Hence, let $g_2(\tau) =  C \sup_{t \in (0,T)}|t^{1-\alm}\varphi(t)| \tau^{\alm-1} \frac{1}{\n{\frac{\omega}{2}}^{\frac{2}{\alm}}     }$. It is integrable on $(t-\delta,t)$, because
\begin{align*}
    \int_{t-\delta}^{t} \tau^{\alm-1}\dta &\leq  \intzt \tau^{\alm-1}\dta = \frac{t^{\alm}}{\alm}.
\end{align*}
\end{enumerate}
Thus, we have an integrable majorant and according to Theorem \ref{lir} we get \eqref{lqwa}.

\end{proof}

The following technical lemma will also be needed in the proof of Theorem \ref{lemma1}.
\begin{lemma}
Let $\delta \in (0,t)$ and $m \in \N$. Then, for any $\lf_1, \ldots, \lf_m > 0$ we have
\begin{equation}\label{Jind}
    \begin{split}
         &\indt S^{0}_m(t-\tau;-\lf_1 ,\ldots, -\lf_m ;-\al_1,\ldots, -\al_m)   d\tau\\
         &= \int_{\delta_m^m}^{\infty}   W\n{-\zeta_m; -\al_m, 1-\al_m  } \times \ldots \times   \int_{\delta_1^m}^{\infty} W\n{-\zeta_1;-\al_1,1-\al_1 } d\zeta_1  \ldots   d\zeta_m ,
\end{split}
\end{equation}
where
\begin{align}\label{deli}
    \delta_i^m := \frac{\lf_i }{\n{\delta - \displaystyle{\sum_{k=i+1}^{m}} \n{\frac{\lf_k }{\zeta_k}}^{\frac{1}{\al_k}} }^{\al_i}   }  \quad \text{for} \ i=1,\ldots,m.
\end{align}
\end{lemma}

\begin{proof}
We prove \eqref{Jind} by induction. \\

\textbf{Base step:} we show \eqref{Jind} for $m=1$. \\
Using representation \eqref{smim} and \eqref{smim1} with $\mu_1 = 0$ we have
\begin{align*}
    &\indt S^{0}_1(t-\tau;-\lf_1;-\al_1)   d\tau = \indt  h_1(t-\tau)  d\tau  \\
 &=  \indt  (t-\tau)^{ - 1} W\n{-\lf_1  (t-\tau)^{-\al_1}; -\al_1, 0  }  d\tau  .
\end{align*}
Applying \eqref{qwj} and substituting the variable $\tau$ with $\zeta_1 = \frac{\lf_1 }{\n{t-\tau}^{\al_1}}$ we obtain
\begin{align*}
    &\indt S^{0}_1(t-\tau;-\lf_1 ;-\al_1)   d\tau  \\
 &=  \indt  \lf_1 \al_1 (t-\tau)^{-\al_1 - 1} W\n{-\lf_1  (t-\tau)^{-\al_1}; -\al_1, 1-\al_1  }  d\tau  \\
    &=  \int_{\frac{\lf_1}{\delta^{\al_1}   }       }^{\infty} W\n{-\zeta_1;-\al_1,1-\al_1 } d\zeta_1 .
\end{align*}
Having in mind notation \eqref{deli} we get \eqref{Jind} for $m=1$. \\

\textbf{Induction step:} we show that for all $m \in \N$ there holds the following implication
\begin{equation}\label{instep}
    \begin{split}
        &\indt S^{0}_m(t-\tau;-\lf_1 ,\ldots, -\lf_m ;-\al_1,\ldots, -\al_m)   d\tau\\
         &= \int_{\delta_m^m}^{\infty}   W\n{-\zeta_m; -\al_m, 1-\al_m  } \times \ldots \times   \int_{\delta_1^m}^{\infty} W\n{-\zeta_1;-\al_1,1-\al_1 } d\zeta_1  \ldots   d\zeta_m\\
         &\implies \indt S^{0}_{m+1}(t-\tau;-\lf_1 ,\ldots, -\lf_{m+1} ;-\al_1,\ldots, -\al_{m+1})   d\tau\\
         &= \int_{\delta_{m+1}^{m+1}}^{\infty}   W\n{-\zeta_{m+1}; -\al_{m+1}, 1-\al_{m+1}  } \times \ldots \times   \int_{\delta_1^{m+1}}^{\infty} W\n{-\zeta_1;-\al_1,1-\al_1 } d\zeta_1  \ldots   d\zeta_{m+1}.
    \end{split}
\end{equation}

Using representation \eqref{smim} and \eqref{smim1} with $\mu_j = 0$ for $j=1,\ldots,m+1$ we have
\begin{align*}
     &  \indt S^{0}_{m+1}(t-\tau;-\lf_1 ,\ldots, -\lf_{m+1} ;-\al_1,\ldots, -\al_{m+1})   d\tau\\
 &=\indt \n{h_1 \ast \ldots \ast h_m \ast h_{m+1}}(t-\tau)   d\tau\\
 &=\indt \int_{0}^{t-\tau} \n{h_1 \ast \ldots \ast h_m}(t-\tau-p_{m+1}) h_{m+1}(p_{m+1}) dp_{m+1}  d\tau\\
 &=\indt \int_{0}^{t-\tau} \n{h_1 \ast \ldots \ast h_m}(t-\tau-p_{m+1}) \\
 &\;\;\;\times p_{m+1}^{ - 1} W\n{-\lf_{m+1} p_{m+1}^{-\al_{m+1}}; -\al_{m+1}, 0  }dp_{m+1}  d\tau.
\end{align*}
Changing the order of integration and substituting $\tau_1 = \tau+p_{m+1} $ we get
\begin{align*}
     &  \indt S^{0}_{m+1}(t-\tau;-\lf_1,\ldots, -\lf_{m+1} ;-\al_1,\ldots, -\al_{m+1})   d\tau\\
 &=\int_{0}^{\delta} p_{m+1}^{ - 1} W\n{-\lf_{m+1}  p_{m+1}^{-\al_{m+1}}; -\al_{m+1}, 0  } \\
 &\;\;\;\times \int_{t-\delta}^{t-p_{m+1}} \n{h_1 \ast \ldots \ast h_m}(t-\tau-p_{m+1})  d\tau dp_{m+1}  \\
  &=\int_{0}^{\delta} p_{m+1}^{ - 1} W\n{-\lf_{m+1}  p_{m+1}^{-\al_{m+1}}; -\al_{m+1}, 0  } \\
 &\;\;\;\times \int_{t-\n{\delta-p_{m+1}}}^{t} \n{h_1 \ast \ldots \ast h_m}(t-\tau_1)  d\tau_1 dp_{m+1} .
\end{align*}
Using the induction hypothesis with $\delta - p_{m+1}$ instead of $\delta$ we obtain
\begin{align*}
     &  \indt S^{0}_{m+1}(t-\tau;-\lf_1,\ldots, -\lf_{m+1} ;-\al_1,\ldots, -\al_{m+1})   d\tau\\
  &=\int_{0}^{\delta} p_{m+1}^{ - 1} W\n{-\lf_{m+1}  p_{m+1}^{-\al_{m+1}}; -\al_{m+1}, 0  } \\
 &\;\;\;\times \int_{\tilde{\delta}_m^m}^{\infty}   W\n{-\zeta_m; -\al_m, 1-\al_m  } \times \ldots \times   \int_{\tilde{\delta}_1^m}^{\infty} W\n{-\zeta_1;-\al_1,1-\al_1 } d\zeta_1  \ldots   d\zeta_m dp_{m+1} ,
\end{align*}
where
\begin{align*}
    \tilde{\delta}_i^m := \frac{\lf_i}{\n{(\delta - p_{m+1})-\displaystyle{\sum_{k=i+1}^{m}} \n{\frac{\lf_k}{\zeta_k}}^{\frac{1}{\al_k}} }^{\al_i}   }   \quad \text{for} \ i=1,\ldots,m.
\end{align*}
Using \eqref{qwj} and then substituting the variable $p_{m+1}$ with $\zeta_{m+1} = \frac{\lf_{m+1} }{p_{m+1}^{\al_{m+1}}}$, we obtain
\begin{align*}
     &  \indt S^{0}_{m+1}(t-\tau;-\lf_1,\ldots, -\lf_{m+1};-\al_1,\ldots, -\al_{m+1})   d\tau\\
  &=\int_{0}^{\delta} \al_{m+1}\lf_{m+1}  p_{m+1}^{-\al_{m+1}-1} W\n{-\lf_{m+1}  p_{m+1}^{-\al_{m+1}}; -\al_{m+1}, 1- \al_{m+1} } \\
 &\;\;\;\times \int_{\tilde{\delta}_m^m}^{\infty}   W\n{-\zeta_m; -\al_m, 1-\al_m  } \times \ldots \times   \int_{\tilde{\delta}_1^m}^{\infty} W\n{-\zeta_1;-\al_1,1-\al_1 } d\zeta_1  \ldots   d\zeta_m dp_{m+1} \\
   &=\int_{\frac{\lf_{m+1} }{\delta^{\al_{m+1}}}      }^{\infty}  W\n{-\zeta_{m+1} ; -\al_{m+1}, 1- \al_{m+1} } \\
 &\;\;\;\times \int_{\delta_m^{m+1}}^{\infty}   W\n{-\zeta_m; -\al_m, 1-\al_m  } \times \ldots \times   \int_{\delta_1^{m+1}}^{\infty} W\n{-\zeta_1;-\al_1,1-\al_1 } d\zeta_1  \ldots   d\zeta_m d\zeta_{m+1},
\end{align*}
because if $\zeta_{m+1} = \frac{\lf_{m+1} }{p_{m+1}^{\al_{m+1}}}$
then for $i=1,\ldots,m$ we have
\begin{align*}
    \tilde{\delta}_i^m &= \frac{\lf_i }{\n{\delta - \n{\frac{\lf_{m+1}}{\zeta_{m+1}}}^{\frac{1}{\al_{m+1}}}  -\displaystyle{\sum_{k=i+1}^{m}} \n{\frac{\lf_k}{\zeta_k}}^{\frac{1}{\al_k}} }^{\al_i}   }   \\
    &=\frac{\lf_i }{\n{\delta  -\displaystyle{\sum_{k=i+1}^{m+1}} \n{\frac{\lf_k }{\zeta_k}}^{\frac{1}{\al_k}} }^{\al_i}   }  = \delta_i^{m+1}.
\end{align*}
Therefore, we get \eqref{instep} and hence we have \eqref{Jind}.
\end{proof}
Now we can proceed to the proof of the jump relation for the integral operator corresponding to the kernel $E(x,t)$. 

\begin{proof}[\textbf{Proof of Theorem \ref{lemma1}}]
By Lemma \ref{duis} we can differentiate under the integral sign, hence for $x \neq \st$
we obtain
\begin{align*}
    \frac{\partial }{\partial x }  \intzt \varphi(\tau) E(x-s(\tau),t-\tau) \dta =   \intzt  \varphi(\tau) E_x(x-s(\tau),t-\tau)\dta .
\end{align*}
Set
\begin{align*}
    L=\intzt \varphi(\tau) E_x(s(t) - s(\tau),t-\tau) \dta - \intzt  \varphi(\tau) E_x(x-s(\tau),t-\tau) \dta.
\end{align*}
Our goal is to show that 
\eqq{\lim_{x\rightarrow s(t)^{\mp}} \n{L\pm \jd \varphi(t)}=0.}{qwh}
For $\de\in (0,t)$, which will be chosen later, we define 
\eqq{
L\equiv  L_{1}+\varphi(t)M+L_{2}\equiv 
}{qek}
\[
=\indt  [\varphi(\tau)- \varphi(t)]\nk{\gdxtta - \gdxxta}\dta
\]
\[
+\varphi(t)M+ \inztd \varphi(\tau) \nk{\gdxtta - \gdxxta}\dta,
\]
where 
\[
M =  \indt \nk{\gdxtta - \gdxxta}\dta.
\]
This decomposition allows us to isolate the main contribution to the jump: the term $\varphi(t) M$ gives rise to the principal jump $\pm \frac{1}{2}\varphi(t)$ in the limit $x \to s(t)^\mp$, while the terms $L_1$ and $L_2$ will be shown to vanish in the same limit. To analyze the contribution of $M$, we further split it into three parts:
\begin{equation}\label{qwi1}
\begin{split}
     M &= \indt \n{ E_x\n{x-\st,t-\tau}-E_x(x-\sta,t-\tau)        } d\tau\\
    &\;\;\;+ \indt  E_x(s(t)-s(\tau),t-\tau) d\tau - \indt E_x(x-s(t),t-\tau) d\tau\\
    &\equiv M_1 + M_2 + J.
\end{split}
\end{equation}
First, we show that 
\eqq{\forall \de\in (0, t) \hd \hd \lim_{x\rightarrow s(t)^{\mp}} J= \mp \jd.}{qwi}
Using the representation \eqref{emt} and \eqref{fmt} and formulas \eqref{smim} and \eqref{smim1} with $\mu_j = 0$ for $j = 1, \ldots, m$ we have 
\begin{equation}\label{ej0}
    \begin{split}
        E(x-\st,t-\tau) &= \n{4\pi}^{-\frac{1}{2}} \intzn p^{-\frac{1}{2}} \exp\n{-\frac{|x-\st|^2}{4p} } \\
    &\;\;\;\times S^{0}_m(t-\tau;-\lambda_1 p,\ldots, -\lambda_m p;-\al_1,\ldots, -\al_m)   dp,
    \end{split}
\end{equation}
where 
\begin{align*}
 S^{0}_m(t-\tau;-\lambda_1 p,\ldots, -\lambda_m p;-\al_1,\ldots, -\al_m) = (h_1 \ast h_2 \ast \ldots \ast h_m)(t-\tau)
\end{align*}
and
\begin{align*}
    h_j(t) = t^{-1} W(-\lambda_j p t^{-\al_j};-\al_j,0) \quad \text{for} \ j=1,\ldots, m.
\end{align*}

Let us calculate $E_x(x-\st,t-\tau)$ for $x \neq \st$. From \eqref{ej0} we have for $x \neq \st$
\begin{equation}\label{ej}
    \begin{split}
        E_x(x-\st,t-\tau) &= \jd (\st-x) \n{4\pi}^{-\frac{1}{2}} \intzn p^{-\trd} \exp\n{-\frac{|x-\st|^2}{4p} }  \\
    &\;\;\;\times S^{0}_m(t-\tau;-\lambda_1 p,\ldots, -\lambda_m p;-\al_1,\ldots, -\al_m)   dp.
    \end{split}
\end{equation}
Hence, we have 
\begin{align*}
 J &= \jd (x-\st) \indt \n{4\pi}^{-\jd} \intzn p^{-\trd} \exp\n{-\frac{|x-\st|^2}{4p} } \\   
 &\;\;\; \times S^{0}_m(t-\tau;-\lambda_1 p,\ldots, -\lambda_m p;-\al_1,\ldots, -\al_m)   dp d\tau.
\end{align*}
Changing the order of integration we obtain
\begin{align*}
 J &= \jd (x-\st)  \n{4\pi}^{-\jd} \intzn p^{-\trd} \exp\n{-\frac{|x-\st|^2}{4p} } \\   
 &\;\;\; \times \indt S^{0}_m(t-\tau;-\lambda_1 p,\ldots, -\lambda_m p;-\al_1,\ldots, -\al_m)   d\tau dp.
\end{align*}
Changing the variable $p \rightarrow z = \frac{p}{|x-\st|^2}$ we get
\begin{equation}\label{ejj}
\begin{split}
     J &= \jd \sign(x-\st)  \n{4\pi}^{-\jd} \intzn z^{-\trd}  \exp\n{-\frac{1}{4z} } \\   
 &\;\;\; \times \indt S^{0}_m(t-\tau;-\lambda_1 z|x-\st|^2,\ldots, -\lambda_m z|x-\st|^2;-\al_1,\ldots, -\al_m)   d\tau  dz .
\end{split}
\end{equation}

Using \eqref{Jind}, with $\lf_i = \lambda_i z|x-\st|^2$ for $i = 1, \ldots m$, in \eqref{ejj} we obtain
\begin{equation}\label{qwj1}
    \begin{split}
         J &= \jd  \sign(x-\st)    \n{4\pi}^{-\jd} \intzn z^{-\trd}  \exp\n{-\frac{1}{4z} } \\
         &\;\;\;\times \int_{\delta_m^m}^{\infty}   W\n{-\zeta_m; -\al_m, 1-\al_m  } \times \ldots \times  \int_{\delta_1^m}^{\infty} W\n{-\zeta_1;-\al_1,1-\al_1 } d\zeta_1 \ldots   d\zeta_m  d z.
\end{split}
\end{equation}
Let us recall (see \eqref{deli}) that $\delta_i^m$ depends on $x$ for $i=1,\ldots,m$. Using dominated convergence theorem we can go to the limit with $x \to \st^{\mp}$ under the integral sign in \eqref{qwj1}. Indeed, using \eqref{qwl1}, the fact that $\delta_i^m > 0$ for $i=1,\ldots,m$ and \eqref{qwl} we have 
\begin{align*}
    &\dm{ \int_{\delta_m^m}^{\infty}   W\n{-\zeta_m; -\al_m, 1-\al_m  } \times \ldots \times   \int_{\delta_1^m}^{\infty} W\n{-\zeta_1;-\al_1,1-\al_1 } d\zeta_1  \ldots   d\zeta_m} \\
    &\leq \int_{0}^{\infty}  W\n{-\zeta_m; -\al_m, 1-\al_m  } \times \ldots \times   \int_{0}^{\infty} W\n{-\zeta_1;-\al_1,1-\al_1 } d\zeta_1  \ldots   d\zeta_m \\
    &=1.
\end{align*}
Thus, we may consider the function $g(z) := z^{-\trd} \exp 
\n{- \frac{1}{4z}   }  $ as the majorant, which is integrable since
\begin{equation}\label{qwl2}
\begin{split}
    \intzn z^{-\trd} \exp \n{- \frac{1}{4z}   }  d z &\stackrel{q = \frac{1}{4z}}{=} \intzn 4^{\trd} q^{\trd} e^{-q} \frac{1}{4q^2} dq \\
&=  2\intzn q^{-\jd} e^{-q}  dq= 2 \Gamma\n{\jd} = 2\sqrt{\pi} .
\end{split}
\end{equation}
Moreover, using the fact that $\lim_{x \to\st} \delta_i^m = 0$ for $i=1,\ldots,m$ and the absolute continuity of Lebesgue integral we have
\begin{align*}
    &\lim_{x \to \st} \int_{\delta_m^m}^{\infty}   W\n{-\zeta_m; -\al_m, 1-\al_m  } \times \ldots \times   \int_{\delta_1^m}^{\infty} W\n{-\zeta_1;-\al_1,1-\al_1 } d\zeta_1  \ldots   d\zeta_m \\
    &= \int_{0}^{\infty}  W\n{-\zeta_m; -\al_m, 1-\al_m  } \times \ldots \times   \int_{0}^{\infty} W\n{-\zeta_1;-\al_1,1-\al_1 } d\zeta_1  \ldots   d\zeta_m  = 1.
\end{align*}
Hence, going to the limit with $x \to \st^{\mp}$ in \eqref{qwj1} we obtain
\begin{align*}
 \lim_{x \to \st^{\mp}} J &= \mp \jd     \frac{1}{\sqrt{4\pi}} \intzn z^{-\trd} \exp \n{- \frac{1}{4z}   }  d z = \mp \jd,
\end{align*}
where we have used \eqref{qwl2}. Thus, we have \eqref{qwi}. In our further calculations we need the following estimate 
\begin{align}\label{ej1}
    \indt \dm{E_x(x-\st,t-\tau)} \dta \leq \jd.
\end{align}
Using the representation \eqref{ej} and the property \eqref{ej01} we have
\begin{equation*}
    \begin{split}
         \dm{E_x(x-\st,t-\tau)}   &= \jd |\st-x| \frac{1}{\sqrt{4\pi}} \intzn p^{-\trd} \exp\n{-\frac{|x-\st|^2}{4p} }  \\
    &\;\;\;\times S^{0}_m(t-\tau;-\lambda_1 p,\ldots, -\lambda_m p;-\al_1,\ldots, -\al_m)   dp.
    \end{split}
\end{equation*}
Thus, we see that
\begin{align*}
    \dm{E_x(x-\st,t-\tau)} = \sign(\st-x) \cdot E_x(x-\st,t-\tau).
\end{align*}
Hence, using the above equality, the definition \eqref{qwi1} of $J$ and the representation \eqref{qwj1} we have
\begin{align*}
     \indt \dm{E_x(x-\st,t-\tau)} \dta &= \sign(\st-x)  \indt E_x(x-\st,t-\tau) \dta \\
     &= -\sign(\st-x) \cdot J = \sign(x-\st) \cdot J = |J|.
\end{align*}
Therefore, using again \eqref{qwj1}, we can estimate
\begin{align*}
     &\indt \dm{E_x(x-\st,t-\tau)} \dta  \\
     &\leq   \jd     \frac{1}{\sqrt{4\pi}} \intzn z^{-\trd} \exp \n{- \frac{1}{4z}   } \int_{0}^{\infty}   W\n{-\zeta_m; -\al_m, 1-\al_m  }  \\
 &\;\;\;\times \ldots \times \int_{ 0   }^{\infty}  W\n{-\zeta_1;-\al_1,1-\al_1 } d\zeta_1  \ldots  d\zeta_m  d z \\
 &=  \jd     \frac{1}{\sqrt{4\pi}} \intzn z^{-\trd} \exp \n{- \frac{1}{4z}   }  d z = \jd,
\end{align*}
where we also used \eqref{qwl} and \eqref{qwl2}. Thus, we get \eqref{ej1}. In the next step, we shall deal with $M_{1}$ and $M_{2}$. Having in mind \eqref{qwi1}, we have
\begin{align*}
    M_{1} &= \indt \n{ E_x(x-\st,t-\tau)-E_x(x-\sta,t-\tau)} d\tau \\
 &= \indt  \int_{0}^{1} E_{xx}( x-\sta + \mu(\sta-\st), t-\tau) d\mu (\sta - \st) d\tau .
\end{align*}
Using \eqref{ls3} and the assumption on $s$ we obtain
\begin{align*}
    | M_{1}| &\leq C \indt (t-\tau)^{-\aldm-1} |\sta-\st| d\tau \\
    &\leq C \int_{t-\delta}^{t} (t-\tau)^{\beta-\aldm-1}  d\tau \leq C \delta^{\beta-\aldm}.
\end{align*}
Therefore, from the above estimate we obtain
\eqq{|M_{1}|\leq C\de^{\beta- \aldm}}{qec}
and indirectly we get the inequality
\begin{equation}\label{qecc}
        \indt \dm{ E_x(x-\st,t-\tau)-E_x(x-\sta,t-\tau)} d\tau \leq  C\de^{\beta- \aldm},
\end{equation}
which will be needed below.
Using \eqref{ls2} and the h\"older continuity of $s$ we get the estimate for $M_{2}$ 
\begin{align*}
    |M_2| &\leq \indt |E_x(s(t)-s(\tau),t-\tau)|    d\tau \leq C \indt  |\st-\sta| \cdot (t-\tau)^{-\aldm-1}  d\tau \\
    &\leq C \indt (t - \tau)^{\beta - \aldm -1 }\dta = C\de^{\beta - \aldm}.
\end{align*}
Thus, we obtain
\begin{align}\label{qed}
     |M_2| \leq C\de^{\beta - \aldm}
\end{align}
and indirectly we have
\begin{align}\label{qedd}
     \indt |E_x(s(t)-s(\tau),t-\tau)|    d\tau  \leq C\de^{\beta - \aldm}.
\end{align}
Combining \eqref{qwi}, \eqref{qec} and \eqref{qed} we have
\begin{align}\label{me}
    \lim_{x \to \st^{\mp}} \dm{ M  \pm \jd   } \leq C \de^{\beta - \aldm}.
\end{align}
Let us deal with $L_1$. Recalling \eqref{qek} we get
\begin{align*}
    L_1 &=  \indt  [\varphi(\tau)- \varphi(t)]\nk{\gdxtta - \gdxxta}\dta\\
    &=\indt [\varphi(\tau)- \varphi(t)] \bigg{[} \Big{(} E_x\n{x-\st,t-\tau} - E_x(x-\sta,t-\tau)   \Big{)}     \\
    &\;\;\;+  E_x(s(t)-s(\tau),t-\tau) - E_x(x-s(t),t-\tau) \bigg{]} d\tau .
\end{align*}
Hence, we obtain
\begin{align*}  
   |L_1| &\leq \sup_{t-\delta < \tau < t} |\varphi(\tau)- \varphi(t)| \bigg{(} \indt \dm{E_x\n{x-\st,t-\tau} - E_x(x-\sta,t-\tau) } \dta\\
   &\;\;\;+ \indt \dm{E_x(s(t)-s(\tau),t-\tau)} \dta + \indt \dm{E_x(x-s(t),t-\tau)} \dta \bigg{)}.
\end{align*}
Using \eqref{ej1}, \eqref{qecc} and \eqref{qedd} we get
\begin{align}\label{qed2}
   \lim_{x \to \st} |L_1| &\leq C\n{\de^{\beta - \aldm} + 1 } \sup_{t-\delta < \tau < t} |\varphi(\tau)- \varphi(t)|.
\end{align}
Finally, to estimate $L_{2}$ we write 
\[
L_{2}= \inztd \varphi(\tau)\int_{0}^{1} E_{xx}(x - \sta + \mu(\st-x), t- \tau)d\mu (\st-x)\dta .
\]
Next, applying \eqref{ls3}, we obtain 
\begin{equation*}
    \begin{split}
  |L_{2}| &\leq C \sup_{t \in (0,T)}|t^{1-\alm} \varphi(t)| \inztd \tau^{\alm-1} (t-\tau)^{-\aldm-1} \dta |x- \st | \\
  &\leq C \sup_{t \in (0,T)}|t^{1-\alm} \varphi(t)| \de^{- \aldm -1}t^{\alm} |x- \st |.
    \end{split}
\end{equation*}
Thus, we get
\begin{align}\label{qeh1}
    \lim_{x \to \st} L_2 = 0.
\end{align}
Combining \eqref{me}, \eqref{qed2} and \eqref{qeh1} we have
\begin{align*}
     \lim_{x \to \st^{\mp}} \dm{L_1 + \varphi(t) M + L_2 \pm \jd \varphi(t)} \leq C\n{\de^{\beta - \aldm} + 1 } \sup_{t-\delta < \tau < t} |\varphi(\tau)- \varphi(t)| + C\de^{\beta - \aldm} .
\end{align*}
Since the left-hand side is independent of $\delta$, and since the right-hand side can be made arbitrary small if $\de$ is sufficiently small, we get \eqref{qwh}. 

\end{proof}

\section{Continuity of the integral operator corresponding to the kernel Z}

In this chapter, we prove Theorem \ref{lemma2}. First, we show an auxiliary lemma that allows us to differentiate under the integral sign in \eqref{qwa2}.

\begin{lemma}\label{duis2}
    Let $t^{1-\alm} \varphi(t)\in C[0,T]$. Assume that there exists $\beta>\frac{\alm}{2}$ such that $|s(t)- s(\tau)|\leq |t- \tau|^{\beta}$ for $t, \tau \in [0, T]$. Then, for each $t\in (0,T]$ and $x \neq \st$ there holds
\eqq{ \frac{\partial }{\partial x }  \intzt \varphi(\tau) Z(x-s(\tau),t-\tau) \dta =   \intzt  \varphi(\tau) Z_x(x-s(\tau),t-\tau)\dta .}{lqwa2}
\end{lemma}

\begin{proof}
We are going to use Theorem \ref{lir} in order to prove \eqref{lqwa2}. First, we show that $\varphi(\tau) Z(x-s(\tau),t-\tau)$ is a Lebesgue integrable function of $\tau$ for each $t \in (0,T]$ and $x \in \rr$ such that $x \neq \st$. In fact, using \eqref{ls4} and the assumption on $\varphi$, we get
\begin{align*}
    \intzt \dm{\varphi(\tau) Z(x-s(\tau),t-\tau)} \dta &\leq C \sup_{t \in (0,T)}|t^{1-\alm}\varphi(t)| \intzt \tau^{\alm-1} \n{t-\tau}^{-\aldm}  \dta \\
   &=  C \sup_{t \in (0,T)}|t^{1-\alm}\varphi(t)| \cdot t^{\aldm} B\n{\alm, 1-\aldm}.
\end{align*}

Moreover, it is obvious that for almost all $\tau \in (0,T)$ the partial derivative $\varphi(\tau) Z_x(x-s(\tau),t-\tau)$ exists for all $x \in \rr$ and $x \neq s(t)$. According to point 3. of Theorem \ref{lir} we must find an integrable function $g: (0,T) \to \rr$ such that 
\begin{align*}
    |\varphi(\tau) Z_x(x-s(\tau),t-\tau)| \leq g(\tau) \quad \text{for almost every} \ \tau \in (0,T).
\end{align*}
Using \eqref{ls5}, we have
\begin{align*}
    | \varphi(\tau) Z_{x}(x-s(\tau) ,t-\tau)| \leq C \sup_{t \in (0,T)}|t^{1-\alm}\varphi(t)| \cdot \tau^{\alm-1} (t-\tau)^{-\alm}.
\end{align*}
Hence, let $g(\tau):= C \sup_{t \in (0,T)}|t^{1-\alm}\varphi(t)| \cdot \tau^{\alm-1} (t-\tau)^{-\alm}$ be the majorant for the function $\varphi(\tau) Z_x(x-\sta,t-\tau)$. Let us show that $g$ is integrable in $(0,t)$
\begin{align*}
     \intzt  |g(\tau)| \dta &=  C \sup_{t \in (0,T)}|t^{1-\alm}\varphi(t)| \intzt \tau^{\alm-1} (t-\tau)^{-\alm} \dta\\
     &=  C \sup_{t \in (0,T)}|t^{1-\alm}\varphi(t)| B\n{\alm,1-\alm} < \infty.
\end{align*}
Thus, we have an integrable majorant and according to Theorem \ref{lir} we get \eqref{lqwa2}.

\end{proof}

Now we can proceed to the proof of continuity of the integral operator corresponding to the kernel $Z$. 

\begin{proof}[\textbf{Proof of Theorem \ref{lemma2}}]
Fix any $t \in (0,T]$ and $\ep > 0$. Let $\delta \in (0,t)$ be chosen below. According to Lemma \ref{duis2}, for any $x \neq \st$, we find
\begin{align*}
    \frac{\partial }{\partial x }  \intzt \varphi(\tau) Z(x-s(\tau),t-\tau) \dta =   \intzt  \varphi(\tau) Z_x(x-s(\tau),t-\tau)\dta .
\end{align*}
We want to show that
\begin{align}\label{lm2}
   \lim_{x \to s(t)} \n{ \intzt \varphi(\tau) Z_{x}(x-\sta,t-\tau)d\tau - \intzt \varphi(\tau) Z_{x}(\st-\sta,t-\tau)d\tau} = 0.
\end{align}
We have
\begin{align*}
    &\intzt \varphi(\tau) Z_{x}(x-\sta,t-\tau)d\tau - \intzt \varphi(\tau) Z_{x}(\st-\sta,t-\tau)d\tau\\
    &= \indt \varphi(\tau) \n{Z_{x}(x-\sta,t-\tau)- Z_{x}(\st-\sta,t-\tau) }d\tau\\
    &\;\;\;+ \inztd \varphi(\tau) \n{ Z_{x}(x-\sta,t-\tau) - Z_{x}(\st-\sta,t-\tau)       } d\tau\\
    &= Z_1 + Z_2.
\end{align*}
Using \eqref{ls5} we can estimate
\begin{align*}
    |Z_1| &\leq \indt |\varphi(\tau)| \n{|Z_{x}(x-\sta,t-\tau)|+ |Z_{x}(\st-\sta,t-\tau)| }d\tau\\
    &\leq C  \sup_{t\in (0,T)} |t^{1-\alm} \varphi(t)| \indt \tau^{\alm-1} (t-\tau)^{-\alm} d\tau\\
    &\leq C  \sup_{t\in (0,T)} |t^{1-\alm} \varphi(t)| \cdot (t-\delta)^{\alm-1}  \indt (t-\tau)^{-\alm} d\tau\\
    &=C  \sup_{t\in (0,T)} |t^{1-\alm} \varphi(t)| \cdot (t-\delta)^{\alm-1} \frac{\delta^{1-\alm}}{1-\alm} \\
    &=  C  \sup_{t\in (0,T)} |t^{1-\alm} \varphi(t)| \cdot  \frac{\delta^{1-\alm}}{(t-\delta)^{1-\alm}} \leq \frac{\ep}{2}
\end{align*}
for suffciently small $\delta$. Henceforth $\delta$ is fixed. In order to estimate $Z_2$ we write
\begin{align*}
  Z_2 = \inztd \varphi(\tau) \int_{0}^{1}  Z_{xx}(\mu(x-\sta) + (1-\mu)(\st-\sta),t-\tau)   d\mu (x-\st)d\tau 
\end{align*}
and using \eqref{ls6} we can estimate
\begin{align*}
    |Z_2| &\leq C \sup_{t\in (0,T)} |t^{1-\alm} \varphi(t)| \inztd  \tau^{\alm-1}(t-\tau)^{-\trd \alm} d\tau |x-\st|\\
   &\leq C  \sup_{t\in (0,T)} |t^{1-\alm} \varphi(t)| \cdot \delta^{-\trd \alm} \inztd  \tau^{\alm-1} d\tau  |x-\st|\\
    &= C \sup_{t\in (0,T)} |t^{1-\alm} \varphi(t)|  \cdot \delta^{-\trd \alm}  \frac{(t-\delta)^{\alm}}{\alm}  |x-\st|\\
     &\leq C  \sup_{t\in (0,T)} |t^{1-\alm} \varphi(t)|  \frac{t^{\alm}}{\delta^{\trd \alm}}  |x-\st| \leq \frac{\ep}{2}
\end{align*}
for all $x$ such that the difference $|x-\st|$ is sufficiently small.
Thus, we obtain \eqref{lm2}.
\end{proof}

\section{Acknowledgments}

I am deeply grateful to Prof. Mykola Krasnoshock for the inspiring idea that helped me to overcome a difficulty in one of the proofs. I also wishes to thank Prof. Arsen Pskhu for helpful clarifications regarding his work. Finally, I would like to thank Prof. Adam Kubica for his valuable remarks and feedback.

 \end{document}